\documentclass[english,10pt]{amsart}

\usepackage{amsmath, amssymb, amsthm}
\usepackage[all]{xy}
\usepackage[activeacute]{babel}

\newtheorem{thm}{Theorem}[section]
\newtheorem{lem}[thm]{Lemma}
\newtheorem{prop}[thm]{Proposition}

\newtheorem{defi}[thm]{Definition}

\newtheorem{rmk}[thm]{Remark}

\newcommand{\stablehomotopy}{\mathcal{SH}}
\newcommand{\spherespectrum}{\mathbf{1}}
\newcommand{\Hom}{\mathrm{Hom}}

\numberwithin{equation}{section}

\begin{document}


\title{On the Orientability of the Slice Filtration}


\author{Pablo Pelaez}
\address{Universit\"at Duisburg-Essen, Mathematik, 45117 Essen, Germany}
\email{pablo.pelaez@uni-due.de}


\subjclass[2000]{Primary 14, 55}



\begin{abstract}
	Let $X$ be a Noetherian separated scheme of finite Krull
	dimension.  We show that the layers of the slice filtration
	in the motivic stable homotopy category $\stablehomotopy$
	are strict modules over Voevodsky's algebraic cobordism spectrum.  
	We also show that the zero slice of any commutative ring spectrum in $\stablehomotopy$
	is an oriented ring spectrum in the sense of Morel, and that its associated
	formal group law is additive.  As a consequence, we get
	that with rational coefficients the slices are in fact
	motives in the sense of Cisinski-D{\'e}glise \cite{mixedmotives}, and 
	have transfers if the
	base scheme is excellent.
	This proves a conjecture
	of Voevodsky \cite[conjecture 11]{MR1977582}.
\end{abstract}


\maketitle


\section{Introduction}
		\label{Introd}
		
	Let $X$ be a Noetherian separated scheme of finite Krull dimension, and
$\mathcal M _{X}$ be
the category of pointed simplicial presheaves in the smooth Nisnevich site $Sm_{X}$ over $X$
equipped with the Quillen model structure \cite{MR0223432} introduced by Morel-Voevodsky \cite{MR1813224}.
We define $T$ in $\mathcal M _{X}$ as  the pointed simplicial presheaf represented by
$S^{1}\wedge \mathbb G _{m}$, where $\mathbb G _{m}$ is the multiplicative group $\mathbb A ^{1}_{X}- \{ 0 \}$ pointed by $1$,
and $S^{1}$ denotes the simplicial circle.  Let
$Spt(\mathcal M _{X})$ denote the category
of symmetric $T$-spectra on $\mathcal M _{X}$
equipped with Jardine's motivic model structure \cite{MR1787949}.  The homotopy category of $Spt (\mathcal M _{X})$
is a triangulated category which will be denoted by $\stablehomotopy$.

Given an integer $q\in \mathbb Z$, we consider the following family of symmetric $T$-spectra
	$$C^{q}_{eff}=\{ F_{n}(S^{r}\wedge \mathbb G _{m}^{s}\wedge U_{+}) \mid n,r,s \geq 0; s-n\geq q; U\in Sm_{X}\}
	$$
where $F_{n}$ is the left adjoint to the $n$-evaluation functor
	$$ev_{n}:Spt(\mathcal M _{X}) \rightarrow \mathcal M _{X}
	$$

Voevodsky \cite{MR1977582} defines the slice filtration as the following family of triangulated subcategories of $\stablehomotopy$
	$$\cdots \subseteq \Sigma _{T}^{q+1}\stablehomotopy ^{eff} \subseteq \Sigma _{T}^{q}\stablehomotopy ^{eff}
		\subseteq \Sigma _{T}^{q-1}\stablehomotopy ^{eff} \subseteq \cdots
	$$
where $\Sigma _{T}^{q}\stablehomotopy^{eff}$ is the smallest full triangulated subcategory of $\stablehomotopy$ which contains
$C^{q}_{eff}$ and is closed under arbitrary coproducts.

It follows from the work of Neeman \cite{MR1308405}, \cite{MR1812507} that the inclusion
	$$i_{q}:\Sigma _{T}^{q}\stablehomotopy ^{eff}\rightarrow \stablehomotopy
	$$
has a right adjoint $r_{q}:\stablehomotopy\rightarrow \Sigma _{T}^{q}\stablehomotopy ^{eff}$, and that the following functors
	$$f_{q}:\stablehomotopy \rightarrow \stablehomotopy$$
	$$s_{q}:\stablehomotopy \rightarrow \stablehomotopy$$
are triangulated, where $f_{q}$ is defined as the composition $i_{q}\circ r_{q}$, and $s_{q}$ is characterized by the fact
that for every $E\in Spt(\mathcal M _{X})$, we have the following distinguished triangle in $\stablehomotopy$
	$$\xymatrix{f_{q+1}E \ar[r]^-{\rho _{q}^{E}}& f_{q}E \ar[r]^-{\pi _{q}^{E}}& s_{q}E \ar[r]& \Sigma _{T}^{1,0}f_{q+1}E}
	$$
We will refer to $f_{q}E$ as the $(q-1)$-connective cover of $E$, and to $s_{q}E$ as the $q$-slice of $E$.
It follows directly from the construction that the $q$-slice of $E$ is right orthogonal with respect to
$\Sigma _{T}^{q+1}\stablehomotopy ^{eff}$, i.e.
	$$\Hom _{\stablehomotopy}(K, s_{q}E)=0
	$$
for every $K$ in $\Sigma _{T}^{q+1}\stablehomotopy^{eff}$.

\section{Strict $MGL$-modules}
		\label{sect-1}

In this section we will show that all the slices have a canonical structure of strict modules
in $Spt (\mathcal M _{X})$ over Voevodsky's algebraic cobordism spectrum.

Let $A$ be a cofibrant ring spectrum with unit in $Spt (\mathcal M _{X})$, and $A\text{-} \mathrm{mod}$ be the category
of left $A$-modules in $Spt (\mathcal M _{X})$.  The work of Jardine \cite[proposition 4.19]{MR1787949}
and Hovey \cite[corollary 2.2]{Hovey:fk} implies that the adjunction
	$$(A\wedge -,U,\varphi): Spt (\mathcal M _{X}) \rightarrow A\text{-} \mathrm{mod}
	$$
induces a Quillen model structure $Spt^{A} (\mathcal M _{X})$ in $A\text{-} \mathrm{mod}$, this means that
a map $f:M\rightarrow N$ in $Spt^{A} (\mathcal M _{X})$ is a weak equivalence or a fibration if and only if
$Uf$ is a weak equivalence or a fibration in $Spt (\mathcal M _{X})$.

It is easy to see that the homotopy category $\stablehomotopy^{A}$ of $Spt^{A} (\mathcal M _{X})$ is
a triangulated category \cite[proposition 3.5.3]{thesis}.

\begin{defi}
		\label{def.effective}
	Let $E$ be a spectrum in $\stablehomotopy$.  We say that $E$ is \emph{effective}
	if $E$ belongs to the triangulated category $\Sigma _{T}^{0}\stablehomotopy ^{eff}$ defined above.
\end{defi}

\begin{thm}
		\label{thm.changecoeffs}
	Let $A$ be an effective cofibrant ring spectrum with unit $u^{A}:\spherespectrum \rightarrow A$
	in $Spt (\mathcal M _{X})$.
	If $s_{0}(u^{A})$ is an isomorphism in $\stablehomotopy$, then for every $q\in \mathbb Z$
	the functor
		$$s_{q}:\stablehomotopy \rightarrow \stablehomotopy
		$$
	factors (up to a canonical isomorphism) through $\stablehomotopy^{A}$ 
		$$\xymatrix{\stablehomotopy \ar[r]^-{s_{q}} \ar@{-->}[dr]_-{\tilde{s}_{q}}& \stablehomotopy \\
								& \stablehomotopy ^{A} \ar[u]_-{UR^{A}}}
		$$
	where $R^{A}$ denotes a fibrant replacement functor in $Spt^{A} (\mathcal M _{X})$.
\end{thm}
\begin{proof}
	This follows directly from \cite[theorem 3.6.20 and lemma 3.6.21]{thesis} or 
	\cite[theorem 2.1(vi)]{MR2576905}.
\end{proof}
	
The following proposition was proved by Voevodsky  \cite[p. 10 section 3.4]{MR1977582}
and Spitzweck \cite[corollaries 3.2 and 3.3]{slice-spitzweck}.

\begin{prop}[Voevodsky]
		\label{prop.unitMGL-iso}
	Let $MGL$ denote Voevodsky's algebraic cobordism spectrum \cite{MR1648048}.
	We have that $MGL$ is effective and its unit map $u^{MGL}:\spherespectrum \rightarrow MGL$
	induces an isomorphism on the zero slices in $\stablehomotopy$
		$$\xymatrix{s_{0}(u^{MGL}):s_{0}\spherespectrum \ar[r]^-{\cong}& s_{0}MGL}
		$$
\end{prop}

Now we can state the main result of this section.

\begin{thm}	
		\label{thm.orientability}
	Let $q \in \mathbb Z$ denote an arbitrary integer and $E$ denote
	an arbitrary symmetric $T$-spectrum in $Spt (\mathcal M _{X})$. 
	We have that the $q$-slice $s_{q}E$ of $E$ is equipped with a canonical structure of
	$MGL$-module in $Spt (\mathcal M _{X})$.  This implies that
	over any base scheme, the slices are always oriented cohomology theories
	in the sense of D\'eglise \cite[example 2.12(2)]{MR2466188}.
\end{thm}
\begin{proof}
	This follows immediately from theorem \ref{thm.changecoeffs}
	and proposition \ref{prop.unitMGL-iso}.
\end{proof}

\begin{rmk}
		\label{rmk.consequences-orient}
	One of the interesting consequences of theorem \ref{thm.orientability}
	is the fact that over any Noetherian separated base scheme of finite
	Krull dimension,
	once we pass to the slices 
	it is possible to apply all the formalism
	developed by D\'eglise in \cite{MR2466188},
	e.g. Chern classes and the Gysin triangle.
\end{rmk}

\section{Oriented Ring Spectra and Formal Group Laws}
		\label{sect-2}
		
In this section we will show that given a commutative ring spectrum $E$ in $\stablehomotopy$, its zero slice $s_{0}E$
is an oriented ring spectrum (in the sense of Morel \cite[definition 3.1]{MR1876212})
with additive formal group law
in $\stablehomotopy$.  To simplify the notation we will denote by $\mathbb P ^{n}$ the trivial projective bundle of
rank $n$ over our base scheme $X$.

\begin{defi}
		\label{def.orientedringspectra}
	Let $E$ be a commutative ring spectrum in $\stablehomotopy$
	with unit 
		$$u^{E}:\spherespectrum \rightarrow E$$
	We say that $E$ is an \emph{oriented ring spectrum} if there exists
	an element $x_{E}$ in $\Hom _{\stablehomotopy}(F_{0}(\mathbb P ^{\infty}), S^{1}\wedge \mathbb G _{m}\wedge E)$,
	where $\mathbb P ^{\infty}$ is the colimit of the diagram
		$$\mathbb P ^{1}\rightarrow \mathbb P ^{2}\rightarrow \cdots \rightarrow \mathbb P ^{n} \rightarrow \cdots
		$$
	given by the inclusions of the respective hyperplanes at infinity, such that $x_{E}$ pulls back to
	the following composition
		$$\xymatrix{F_{0}(\mathbb P ^{1})\cong
					F_{0}(S^{1}\wedge \mathbb G _{m}) \ar[rr]^-{id \wedge u^{E}} &&
					F_{0}(S^{1}\wedge \mathbb G_{m})\wedge E \cong S^{1}\wedge \mathbb G_{m}\wedge E}
		$$
\end{defi}

The following proposition is classical.

\begin{prop}[cf. \cite{MR0253350}, \cite{MR0290382}, \cite{MR1876212}]
		\label{prop.FGL}
	Let $(E, x_{E})$ be an oriented ring spectrum in $\stablehomotopy$, and let
	$m:\mathbb P ^{\infty} \times \mathbb P ^{\infty}\rightarrow \mathbb P ^{\infty}$
	be the map induced by the corresponding Segre embeddings.  The pullback of $x_{E}$ along $m$,
	is a formal group law $F_{E}$
		$$F_{E}=\sum _{i+j\geq 1}c_{ij}x^{i}y^{j}
		$$
	where the coefficients $c_{ij}$ are elements in the abelian group
		$$\Hom _{\stablehomotopy}(F_{0}(S^{i+j-1}\wedge \mathbb G_{m}^{i+j-1}),E)
		$$
	and $x$ (resp. $y$) is the pullback of $x_{E}$ along the projection in the first factor 
	$p_{1}:\mathbb P ^{\infty} \times \mathbb P ^{\infty}\rightarrow \mathbb P ^{\infty}$
	(resp. second factor).
\end{prop}

\begin{lem}
		\label{lemma.slice-monoids}
	Let $E$ be a commutative ring spectrum in $\stablehomotopy$
	with unit $u^{E}:\spherespectrum \rightarrow E$, then its zero slice $s_{0}E$ is also
	a commutative ring spectrum in $\stablehomotopy$, and the induced map
		$$s_{0}(u^{E}):s_{0}\spherespectrum \rightarrow s_{0}E
		$$
	is a map of ring spectra in $\stablehomotopy$.
\end{lem}
\begin{proof}
	The fact that $s_{0}E$ is a ring spectrum in $\stablehomotopy$ follows from \cite[theorem 3.6.13]{thesis},
	on the other hand the naturality of the pairings constructed in \cite[theorem 3.6.9]{thesis} implies that $s_{0}E$ is
	also commutative in $\stablehomotopy$ and that $s_{0}(u^{E})$ is a map of ring spectra
	in $\stablehomotopy$.
\end{proof}
	
\begin{lem}
		\label{lemma.zeroMGL-oriented}
	The natural map
		$$\pi _{0}^{MGL}:MGL\cong f_{0}MGL \rightarrow s_{0}MGL
		$$
	is a map of ring spectra in $\stablehomotopy$.
\end{lem}
\begin{proof}
	By proposition \ref{prop.unitMGL-iso} we have that $MGL$ is naturally isomorphic
	to $f_{0}MGL$ in $\stablehomotopy$.
	On the other hand, theorem 3.6.10(3) in \cite{thesis} implies that
	$\pi _{0}^{MGL}$ is a map of ring spectra in $\stablehomotopy$.
\end{proof}

\begin{thm}
		\label{thm.ringspectra-oriented}
	Let $E$ be a commutative ring spectrum in $\stablehomotopy$.  Then its zero slice $s_{0}E$
	is an oriented ring spectrum in the sense of Morel and it is equipped with a canonical orientation
	given by the following composition
		$$\xymatrix{MGL \ar[rr]^-{\pi _{0}^{MGL}} && s_{0}MGL  \ar[rrr]^-{(s_{0}(u^{MGL}))^{-1}} &&&
					s_{0}\spherespectrum \ar[rr]^{s_{0}(u^{E})} && s_{0}E}
		$$
	Furthermore, the associated formal group law $F_{s_{0}E}$ of $s_{0}E$ is additive.
\end{thm}
\begin{proof}
	The universality of $MGL$ (cf. \cite[theorem 4.3]{MR1876212},
	\cite[theorem 2.7]{MR2475610} and \cite[proposition A.2]{MR2496504})
	implies that in order to show that the map defined above gives an orientation
	for $s_{0}E$,
	it is enough to see that all the maps are
	in fact maps of ring spectra in $\stablehomotopy$.
	But this follows directly from proposition \ref{prop.unitMGL-iso}
	together with lemmas \ref{lemma.slice-monoids} and \ref{lemma.zeroMGL-oriented}.
	
	On the other hand, the formal group law of $s_{0}E$
		$$F_{s_{0}E}=\sum _{i+j\geq 1} c_{ij}x^{i}y^{j}
		$$
	has coefficients $c_{ij}$ which by construction are in the
	abelian group 
		$$\Hom _{\stablehomotopy}(F_{0}(S^{i+j-1}\wedge \mathbb{G}_{m}^{i+j-1}), s_{0}E)
		$$
	However, if $i+j>1$ then $F_{0}(S^{i+j-1}\wedge \mathbb{G}_{m}^{i+j-1})$ is automatically in
	$\Sigma _{T}^{1}\stablehomotopy^{eff}$; hence
		$$\Hom _{\stablehomotopy}(F_{0}(S^{i+j-1}\wedge \mathbb{G}_{m}^{i+j-1}), s_{0}E)=0
		$$	
	since $s_{0}E$ is right orthogonal with respect to $\Sigma _{T}^{1}\stablehomotopy^{eff}$.
	
	Therefore, the formal group law of $s_{0}E$
		$$F_{s_{0}E}=\sum _{i+j\geq 1} c_{ij}x^{i}y^{j}=x+y
		$$
	is additive, as we wanted.
\end{proof}
\section{Applications}
		\label{sect-3}
		
In this section we will show that with rational coefficients all the slices 
$s_{q}(E)\otimes \spherespectrum _{\mathbb Q}$ are in a natural way motives
in the sense of Cisinski-D\'eglise \cite{mixedmotives}; as a consequence we will get that
over an excellent base scheme
the presheaves of rational rigid homotopy groups have transfers, this proves a conjecture
of Voevodsky \cite[conjecture 11]{MR1977582}.

\subsection{Slices and the Cisinski-D\'eglise category of motives}
		\label{subsect-3.1}
		
	Let 
		$$\mathbf{H}_{\mathbb B ,X}=KGL _{X}^{(0)} \in Spt(\mathcal M _{X})
		$$ 
	denote the Beilinson motivic
	cohomology spectrum constructed by Riou in \cite{MR2651359}.
	The work of Cisinski-D\'eglise shows in particular
	that $\mathbf{H}_{\mathbb B ,X}$ is a commutative
	cofibrant ring spectrum in $Spt(\mathcal M _{X})$ (cf. \cite[corollary 13.2.6]{mixedmotives});
	and that the homotopy category of $\mathbf{H}_{\mathbb B ,X}$-modules 
	$\stablehomotopy ^{\mathbf{H}_{\mathbb B ,X}}$
	is naturally equivalent to the Cisinski-D\'eglise category of motives $DM_{\mathbb B, X}$
	(cf. \cite[theorem 13.2.9]{mixedmotives}).
	
\begin{thm}
		\label{thm.zeroslice-Hbalgebra}
	If we consider rational coefficients,
	the zero slice of the sphere spectrum
	$s_{0}(\spherespectrum)\otimes \spherespectrum _{\mathbb Q}$
	is equipped with a unique structure of
	$\mathbf{H}_{\mathbb B ,X}$-algebra in $Spt (\mathcal M _{X})$.
	
	In particular, there exists a unique map 
	$\eta _{s_{0}(\spherespectrum )\otimes \spherespectrum _{\mathbb Q}}$
	of ring spectra in $\stablehomotopy$
	such that the following diagram is commutative
	
	$$\xymatrix{MGL\otimes \spherespectrum _{\mathbb Q} \ar[rr]^-{\pi _{0}^{MGL}\otimes id} 
					\ar[d]_-{\eta _{\mathbf H}}
					&& s_{0}(MGL)\otimes \spherespectrum _{\mathbb Q}  
					\ar[rrr]^-{(s_{0}(u^{MGL}))^{-1}\otimes id} &&&
					s_{0}(\spherespectrum)\otimes \spherespectrum _{\mathbb Q}  \\
					\mathbf{H}_{\mathbb B ,X} \ar@{-->}[urrrrr]_-{\eta _{s_{0}(\spherespectrum )
					\otimes \spherespectrum _{\mathbb Q}}}&& &&&}
		$$
\end{thm}
\begin{proof}
	By theorem
	\ref{thm.ringspectra-oriented} we have that $s_{0}(\spherespectrum)$ is orientable.
	Therefore, the result follows directly from corollary 13.2.15(Ri),(Rii),(Riii)
	in \cite{mixedmotives}.
\end{proof}

\begin{thm}
		\label{thm.slices-motives}
	Let $q \in \mathbb Z$ denote an arbitrary integer and $E$ denote
	an arbitrary symmetric $T$-spectrum in $Spt (\mathcal M _{X})$. 
	We have that the $q$-slice of $E$ with rational coefficients
	$s_{q}(E)\otimes \spherespectrum _{\mathbb Q}$ is equipped with a canonical structure of
	$\mathbf{H}_{\mathbb B ,X}$-module 
	in $Spt (\mathcal M _{X})$.  This implies that
	over any base scheme, the slices with rational coefficients 
	$s_{q}(E)\otimes \spherespectrum _{\mathbb Q}$
	are motives in the sense of Cisinski-D\'eglise.
\end{thm}
\begin{proof}
	Using corollary 13.2.15(i),(iv),(v) in \cite{mixedmotives}, we get that
	it is enough to show that $s_{q}(E)\otimes \spherespectrum _{\mathbb Q}$
	is a $\mathbf{H}_{\mathbb B ,X}$-module in $\stablehomotopy$.	
	On the other hand, by theorem \ref{thm.zeroslice-Hbalgebra}
	we just need to check that $s_{q}(E)\otimes \spherespectrum _{\mathbb Q}$
	is a $s_{0}(\spherespectrum) \otimes \spherespectrum _{\mathbb Q}$-module in $\stablehomotopy$.	
	Finally, this follows directly from theorem 3.6.14(6) in \cite{thesis}.
\end{proof}
		
\subsection{Rational Rigid Homotopy groups}
		\label{subsect-3.2}
	
	Given a symmetric $T$-spectrum $E$ in $Spt (\mathcal M _{X})$,
	Voevodsky defines the presheaves of rigid homotopy groups $\pi _{p,q}^{rig}(E)$
	on $Sm_{X}$ as follows:
		$$\xymatrix@R=0.5pt{\pi _{p,q}^{rig}(E): Sm _{X} \ar[rr] && \mathrm{Abelian \; Groups} \\
						U \ar@{|->}[rr]&& 
						\Hom _{\stablehomotopy}(F_{0}(S^{p}\wedge \mathbb G _{m}^{q}\wedge U_{+}),
								s_{q}E)} 
		$$
	Conjecture 11 in \cite{MR1977582} claims that these presheaves have transfers.
	
\begin{thm}
		\label{Voevodsky's-conj}
	Let $p, q \in \mathbb Z$ denote arbitrary integers and $E$ denote
	an arbitrary symmetric $T$-spectrum in $Spt (\mathcal M _{X})$.
	Furthermore, assume that the base scheme $X$ is excellent.  Then the 
	presheaves of rigid homotopy groups of $E$ with rational coefficients
	$\pi _{p,q}^{rig}(E)\otimes \mathbb Q$ have transfers.
\end{thm}
\begin{proof}
	Clearly, it suffices to show that $s_{q}(E)\otimes \spherespectrum _{\mathbb Q}$
	has transfers.  Now, theorem \ref{thm.slices-motives}
	implies that $s_{q}(E)\otimes \spherespectrum _{\mathbb Q}$
	is in $DM_{\mathbb B, X}$.  Since we are assuming that $X$ is excellent, theorem 15.1.2
	in \cite{mixedmotives} implies that $DM_{\mathbb B, X}$ is naturally equivalent to the category
	of motives $DM_{\mathrm{qfh}, X}$ constructed using $qfh$-sheaves with rational coefficients.
	Thus, the result follows from theorem 3.3.8 in \cite{MR1403354} (see also proposition 9.5.5 in \cite{mixedmotives})
	which implies than every $qfh$-sheaf is canonically equipped with transfers.
\end{proof}

\begin{rmk}
		\label{rmk.previous-proofs}
	Theorem \ref{Voevodsky's-conj} was proved using the functoriality of the slice filtration
	in \cite[theorem 4.4]{Pelaez:2010fk} for schemes defined over a field of characteristic zero; 
	on the other hand, if the base scheme $X$ is smooth over a perfect field $k$,
	then theorem \ref{Voevodsky's-conj} holds even with integral coefficients
	(cf. \cite[theorem 3.6.22]{thesis}).
	Both proofs rely on the computation of 
	Levine \cite{MR2365658} and Voevodsky \cite{MR2101286} 
	for the zero slice of the sphere spectrum, as well as on the
	work of R{\"o}ndigs-{\O}stv\emph{\ae}r \cite{MR2435654}.
	
	The analogue of this question for the category of $S^{1}$-spectra
	is studied by Levine in \cite{slices-transfers}.
\end{rmk}

\section*{Acknowledgements}
	The author would like to warmly thank Fr\'ed\'eric D\'eglise for several useful conversations and suggestions, 
	as well as for putting in our hands
	the technical tools from \cite{MR2466188} and \cite{mixedmotives}; and also thank Denis-Charles Cisinski for
	bringing to our attention the argument
	which allowed us to extend theorem \ref{Voevodsky's-conj} from
	geometrically unibranch base schemes to arbitrary excellent schemes.

\bibliography{biblio_osf}
\bibliographystyle{abbrv}

\end{document}